\documentclass[12pt,reqno]{article}

\usepackage{amsmath,amssymb,amsthm, mathrsfs}
\usepackage{latexsym}
\usepackage{graphics}
\usepackage{hyperref}
\usepackage{graphicx,psfrag}
\usepackage[bf, small]{titlesec}
\usepackage{authblk} 

\usepackage{lineno}

\theoremstyle{plain}
\newtheorem{Thm}{Theorem}[section]
\newtheorem{Lem}[Thm]{Lemma}
\newtheorem{Prop}[Thm]{Proposition}
\newtheorem{Cor}[Thm]{Corollary}

\theoremstyle{definition}
\newtheorem{Def}[Thm]{Definition}
\newtheorem{Rem}[Thm]{Remark}
\newtheorem{Ex}[Thm]{Example}

\usepackage{standalone}
\usepackage{pgfpages,tikz,pgfkeys,pgfplots}
\usetikzlibrary{arrows,positioning,matrix,fit,backgrounds,shapes,shapes.geometric, intersections}

\usetikzlibrary{shapes.callouts,decorations.pathmorphing} 
\usetikzlibrary{shadows} 
\usepackage{calc} 
\usetikzlibrary{decorations.markings} 
\tikzstyle{vertex}=[circle, draw, inner sep=0pt, minimum size=6pt] 

\usetikzlibrary{arrows,matrix} 

\newcommand{\RR}{\mathbb{R}} 

\newcommand{\PPP}{\mathcal{P}} 
\newcommand{\CCC}{\mathcal{C}} 
\newcommand{\III}{\mathcal{I}} 

\addtocounter{MaxMatrixCols}{20}

\title{Using $p$-row graphs to study $p$-competition graphs}
\author[1]{\small Soogang Eoh}
\author[1]{\small Taehee Hong}
\author[1]{\small Suh-Ryung Kim}
\author[2]{\small Seung Chul Lee}
\affil[1]{\footnotesize Department of Mathematics Education, Seoul National University, Seoul 08826, Republic of Korea}
\affil[2]{\footnotesize Seoul Science High School, Hyehwa-ro 63, Jongno-gu, Seoul 03066, Republic of Korea}
\affil[ ]{\footnotesize\textit{mathfish@snu.ac.kr, ds3mbc@snu.ac.kr, srkim@snu.ac.kr, lschul@sen.go.kr}}

\begin{document}

\maketitle

\begin{abstract}
For a positive integer $p$, the $p$-competition graph of a digraph $D$ is a graph which has the same vertex set as $D$ and an edge between distinct vertices $x$ and $y$ if and only if $x$ and $y$ have at least $p$ common out-neighbors in $D$. A graph is said to be a $p$-competition graph if it is the $p$-competition graph of a digraph. Given a graph $G$, we call the set of positive integers $p$ such that $G$ is a $p$-competition the competition-realizer of a graph $G$. In this paper, we introduce the notion of $p$-row graph of a matrix which generalizes the existing notion of row graph.
We call the graph obtained from a graph $G$ by identifying each pair of adjacent vertices which share the same closed neighborhood the condensation of $G$.
Using the notions of $p$-row graph and condensation of a graph, we study competition-realizers for various graphs to extend results given by Kim {\it et al.}~[$p$-competition graphs,
{\it Linear Algebra Appl.} {\bf 217} (1995) 167--178]. Especially, we find all the elements in the competition-realizer for each caterpillar.
\end{abstract}

\section{Introduction}
Given a digraph $D = (V,A)$, the \emph{competition graph} of $D$ is a simple graph having the same vertex set as $D$ and having an edge $uv$
if for some vertex $x \in V$, the arcs $(u,x)$ and $(v,x)$ are in
$D$. The notion of competition graph is due to Cohen~\cite{cohen1968interval}
and has arisen from ecology.  Competition graphs also have
applications in coding, radio transmission, and modelling of
complex economic systems (see \cite{raychaudhuri1985generalized} and \cite{roberts1999competition}
for a summary of these applications and \cite{greenberg1981graph} for a sample
paper on the modelling application).  Since Cohen introduced the
notion of competition graph, various variations have been defined
and studied by many authors (see the survey articles by Kim~\cite{kim1993competition} and Lundgren~\cite{lundgren1989food}). For recent work on this topic, see~\cite{factor20111, kamibeppu2012sufficient, kuhl2013transversals, li2012competition, zhang20161}.

Kim {\it et al.}~\cite{kim1995p} introduced $p$-competition graphs as a variant of competition graph. For a positive integer $p$, the \emph{$p$-competition graph} $C_p(D)$ of a digraph $D=(V,A)$ is defined to have the vertex set $V$ with an edge between two distinct vertices $x$ and $y$ if and only if, for some distinct vertices $a_1, \ldots, a_p$ in $V$, the pairs $(x, a_1)$, $(y, a_1)$, $(x, a_2)$, $(y, a_2)$, $\ldots$, $(x, a_p)$, $(y, a_p)$ are arcs in $D$. Note that $C_1(D)$ is the ordinary competition graph, which implies that the notion of $p$-competition graph generalizes that of competition graph. A graph $G$ is called a \emph{$p$-competition graph} if there exists a digraph $D$ such that $G=C_p(D)$. By definition, it is obvious that if a nonempty graph $G$ is a $p$-competition graph, then $p \le |V(G)|$.

Competition graphs are closely related to
edge clique covers and the edge clique cover numbers of graphs.
A {\it clique} of a graph $G$ is a subset of
the vertex set of $G$ that induces a complete graph.
We regard an empty set as a clique of $G$ for convenience.
An {\it edge clique cover}
of a graph $G$ is a family of cliques of $G$
such that the end vertices of each edge of $G$
are contained in a clique in the family.
The minimum size of an edge clique cover of $G$ is called
the {\it edge clique cover number}
of $G$, and is denoted by $\theta_e(G)$.
Dutton and Brigham~\cite{dutton1983characterization} characterized a competition graph in terms of its edge clique cover number.

\begin{Thm}[\!\!\cite{dutton1983characterization}]\label{thm:ECC}
A graph $G$ with $n$ vertices is a competition graph if and only if $\theta_e(G) \le n$.
\end{Thm}

A $p$-competition graph $G$ can be characterized in terms of the ``$p$-edge clique cover number" of $G$. For a positive integer $p$, a {\em $p$-edge clique cover} ({\em $p$-ECC} for short) of a graph $G$ is defined to be a multifamily ${\mathcal F}=\{F_1,\ldots,F_r\}$ of subsets of the vertex set of $G$ satisfying the following:
\begin{itemize}
\item For any $J\in\binom{[r]}{p}$, the set $\bigcap_{j \in J} F_j$ is a clique of $G$;
\item The collection $\left\{\bigcap_{j \in J}F_j \mid J\in\binom{[r]}{p}\right\}$ covers all the edges of $G$,
\end{itemize}
where $\binom{[r]}{p}$ denotes the set of $p$-element subsets of the set $\{1,\ldots,r\}$ for a positive integer $r$.
The minimum size $r$ of a $p$-edge clique cover of $G$
is called the {\it $p$-edge clique cover number}  of $G$,
and is denoted by $\theta_e^p(G)$.
The following theorem characterizes $p$-competition graphs and so generalizes Theorem~\ref{thm:ECC}.

\begin{Thm}[\!\!\cite{kim1995p}]\label{thm:pECC}
A graph $G$ with $n$ vertices is a $p$-competition graph if and only if $\theta_e^p(G) \le n$.
\end{Thm}

In this paper, we introduce the notion of $p$-row graph of a matrix which generalizes the existing notion of row graph of a matrix. We also introduce the notion of the condensation of a graph that is obtained from a graph by identifying each pair of adjacent vertices which share the same closed neighborhood. Using these notions, we study competition-realizers for various graphs to extend results given by Kim {\it et al.}~[$p$-competition graphs,
{\it Linear Algebra Appl.} {\bf 217} (1995) 167--178]. Especially, we find all the elements in the competition-realizer for each caterpillar.

\section{$p$-row graphs and competition-realizers}
In this section, we introduce the notion of $p$-row graph of a matrix which generalize the notion of row graph of a matrix and the notion of competition-realizer for a graph. Then we study competition-realizers for various graphs in terms of $p$-row graph and the condensation of a graph. Particularly, we identify the graphs with $n$ vertices whose competition-realizers contain $n$ and $n-1$, respectively.

\begin{Def}
Given a positive integer $p$ and a $(0,1)$-matrix $A$, a graph $G$ is called the \emph{$p$-row graph} of $A$ if the vertices of $G$ are the rows of $A$,
and two vertices are adjacent in $G$ if and only if their corresponding rows have common nonzero entries in at least $p$ columns of $A$.
\end{Def}
If $p=1$, then $G$ is called the \emph{row graph} of $A$, which was introduced by Greenberg {\it et al.}~\cite{greenberg1981graph}.

Suppose that a graph $G$ is a $p$-competition graph with the vertex set $\{v_1,\ldots,v_{n}\}$.
Then there exists a $p$-ECC $\mathcal{F}=\{F_1,\ldots,F_m\}$ of $G$ for a nonnegative integer $m \le n$ by Theorem~\ref{thm:pECC}.
Now we define a square matrix $A=(a_{ij})$ of order $n$ by
\begin{equation}\label{eqn:matrix}
a_{ij}=
\begin{cases}
  1 & \mbox{if } v_i \in F_j; \\
  0 & \mbox{otherwise}.
\end{cases}
\end{equation}
By the definition of $p$-ECC, it is easy to see that $G$ is isomorphic to the $p$-row graph of $A$.
Conversely, suppose that a graph $G$ with $n$ vertices is isomorphic to the $p$-row graph of a square $(0,1)$-matrix $A$ of order $n$. Let
\[F_j = \{v_i \mid a_{ij} = 1\}.\]
and let $\mathcal{F}=\{F_1,\ldots,F_n\}$.
By the definition of $p$-row graph, a vertex $v_s$ and a vertex $v_t$ are adjacent if and only if the $s$th row and the $t$th row of $A$ have common nonzero entries in at least $p$ columns, which is equivalent to the statement that
the pair of vertices $v_s$ and $v_t$ is contained in the sets in $\mathcal{F}$ corresponding to those columns.
Then, by Theorem~\ref{thm:pECC}, $G$ is a $p$-competition graph.

Now we have shown the following statement:

\begin{Thm}\label{thm:rowgraph}
    A graph $G$ with $n$ vertices is a $p$-competition graph if and only if $G$ is isomorphic to the $p$-row graph of a square $(0,1)$-matrix of order $n$.
\end{Thm}

For simplicity's sake, we denote $J_{m,n}$ for the $(0,1)$-matrix of size $m$ by $n$ such that every entry is 1, $I_n$ for the identity matrix of order $n$, and $O_{m,n}$ for the zero matrix of size $m$ by $n$.

For a graph $G$ with $n$ vertices, we denote the set
\[
\{p \in [n]\mid G \mbox{ is a $p$-competition graph} \}
\]
by $\Upsilon(G)$ and call it the \emph{competition-realizer} for $G$.

We make the following simple but useful observations.
\begin{Prop}\label{prop:complete or empty}
Let $G$ be a graph with $n$ vertices.
If $G$ is empty or complete, then $\Upsilon(G)=[n]$.
\end{Prop}
\begin{proof}
If $G$ is empty, then $G$ is a $p$-row graph of $O_{n,n}$ and so, by Theorem~\ref{thm:rowgraph}, is a $p$-competition graph for any $p\in[n]$.
If $G$ is complete, then $G$ is a $p$-row graph of $J_{n,n}$ and so, by the same theorem, is a $p$-competition graph for any $p \in [n]$.
\end{proof}

\begin{Prop}\label{prop:prow}
Given a graph $G$ with $n$ vertices, suppose that $G$ is a $p$-row graph of a matrix of size $n$ by $m$ for positive integers $p$ and $m\leq n$.
Then $\Upsilon(G)\supset\{p+i\mid i\in[n-m]\cup\{0\}\}$.
\end{Prop}
\begin{proof}
Let $M$ be an $n\times m$ matrix whose $p$-row graph is $G$.
Take $i\in[n-m]\cup\{0\}$.
Then we add $i$ all-one columns and $n-m-i$ all-zero columns to $M$ to obtain a square matrix of order $n$ whose $(p+i)$-row graph is $G$.
By Theorem~\ref{thm:rowgraph}, $G$ is a $(p+i)$-competition graph and so $p+i\in \Upsilon(G)$.
\end{proof}

\begin{Prop}\label{prop:add isolated}
Let $G$ be a graph and $G'$ be a graph obtained from $G$ by adding $k$ isolated vertices for a nonnegative integer $k$.
Then $\Upsilon(G')\supset \{p+i\mid i\in[k]\cup\{0\}\}$ for each $p\in\Upsilon(G)$.
\end{Prop}
\begin{proof}
Let $|V(G)|=n$ and suppose $p\in\Upsilon(G)$.
By Theorem~\ref{thm:rowgraph}, $G$ is a $p$-row graph of a square $(0, 1)$-matrix $M$ of order $n$.
Fix $i \in [k] \cup \{0\}$.
We define a square $(0,1)$-matrix $M_i$ of order $n+k$ as follows:
\begin{center}
\begin{tikzpicture}[
    style1/.style={
        matrix of math nodes,
        every node/.append style = {
            text width=#1,
            align=center,
            minimum height=#1
            },
        nodes in empty cells,
        left delimiter = [,
        right delimiter = ],
        column 2/.style={
            nodes={text width=.3*#1}
        }
    }
    ]
    \matrix[style1 = 3ex] (M_i) { & &&& \\ & &&& \\ & &&& \\};
    \draw[dashed] (M_i-1-2.north) -- (M_i-2-2.center);
    \draw[dashed] (M_i-1-4.north) -- (M_i-2-4.center);
    \draw[dashed] (M_i-2-1.west) -- (M_i-2-5.east);

    \node at (M_i-1-1) {$M$};
    \node at (M_i-1-3) {$J_{n,i}$};
    \node at (M_i-1-5) {$O_{n,k-i}$};
    \node at (M_i-3-3) {$O_{k,n+k}$};

    \draw[decorate, decoration={mirror,raise=12pt}] (M_i-2-1.west) node[left=12pt] {$M_i=$};
    \draw[decorate, decoration={mirror,raise=12pt}] ([yshift=-5pt]M_i-2-5.east) node[right=9pt] {.};
\end{tikzpicture}
\end{center}
Obviously, the $(p+i)$-row graph of $M_i$ is $G$ together with $k$ isolated vertices.
By Theorem~\ref{thm:rowgraph}, $p+i\in \Upsilon(G')$.
\end{proof}

For a $p$-row graph $G$ of a matrix $M$ and a vertex $u$ of $G$, we let
\[
\Lambda_M(u)=\{ i \mid \text{the  }i\text{th component of the row corresponding to } u \text{ in } M \text{ is } 1 \}.
\]

A vertex $v$ of a graph $G$ is called a \emph{simplicial vertex} if its neighbors form a clique in $G$.

\begin{Prop}\label{prop:lambda}
Let $G$ be a $p$-row graph of a matrix $M$.
Then, for a non-isolated non-simplicial vertex $u$, $|\Lambda_M(u)|\geq p+1$.
\end{Prop}
\begin{proof}
By the condition on $u$, $u$ is adjacent to two nonadjacent vertices $v$ and $w$.
Then
$$p \leq |\Lambda_M(u)\cap\Lambda_M(v)| \quad\text{ and }\quad p \leq |\Lambda_M(u)\cap\Lambda_M(w)|.$$
Suppose that $|\Lambda_M(u)|\leq p$.
Then
$$|\Lambda_M(u)\cap\Lambda_M(v)|\leq p \quad\text{ and }\quad |\Lambda_M(u)\cap\Lambda_M(w)|\leq p.$$
Thus $|\Lambda_M(u)\cap\Lambda_M(v)|= |\Lambda_M(u)\cap\Lambda_M(w)|=p$ and so $\Lambda_M(u)\cap\Lambda_M(v)= \Lambda_M(u)\cap\Lambda_M(w)=\Lambda_M(u)$.
Hence $\Lambda_M(u) \subset \Lambda_M(v)\cap\Lambda_M(w)$ and so $|\Lambda_M(v)\cap\Lambda_M(w)|\geq p$, which is a contradiction.
\end{proof}

The following proposition characterizes a graph $G$ with $n$ vertices and $n\in\Upsilon(G)$.

We denote a set of $m$ isolated vertices by $\III_m$.
Technically, we let $\III_0=\emptyset$ and $K_0 = \emptyset$.

\begin{Prop} \label{prop:complete}
Let $G$ be a graph with $n$ vertices. Then $G$ is an $n$-competition graph if and only if $G \cong K_m\cup \III_{n-m}$ for some integer $m\in \{0,\ldots,n\}$.
\end{Prop}
\begin{proof}
By definition, $G$ is an $n$-competition graph if and only if $n \in \Upsilon(G)$.
To show the ``if" part, suppose that $G\cong K_m\cup \III_{n-m}$ for some integer $m$, $0 \le m \le n$.
By Proposition~\ref{prop:complete or empty}, $m\in\Upsilon(K_m)$.
By Proposition~\ref{prop:add isolated}, $m+(n-m)\in\Upsilon(G)$.
To show the ``only if" part, suppose that $G$ is an $n$-competition graph.
Then $G$ is isomorphic to the $n$-row graph of a matrix $M$ of order $n$ by Theorem~\ref{thm:rowgraph}.
Suppose that $u$ is a non-isolated vertex in $G$.
Then $u$ is adjacent to a vertex $v$ in $G$, so $|\Lambda_M(u) \cap \Lambda_M(v)|=n$.
Thus we may conclude that each row of $M$ corresponding to a non-isolated vertex is the all-one vector in $\RR^n$.
Thus the subgraph of $G$ induced by non-isolated vertices is a clique.
Hence $G=K_m\cup \III_{n-m}$ where $m$ is the number of non-isolated vertices in $G$.
\end{proof}

\begin{Cor}\label{cor:K_n}
Suppose that a graph $G$ with $n$ vertices has no isolated vertices.
Then $G$ is an $n$-competition graph if and only if $G=K_n$.
\end{Cor}

\begin{Cor}\label{cor:upsilon_n}
Let $G$ be a graph with $n$ vertices.
Then $\Upsilon(G)=[n]$ if and only if $G\cong K_m \cup \III_{n-m}$ for some $m\in \{0,\ldots,n\}$.
\end{Cor}
\begin{proof}
The ``only if'' part immediately follows by Proposition~\ref{prop:complete}.
To show the ``if'' part, suppose that $G\cong K_m\cup \III_{n-m}$ for some $m$, $0\leq m\leq n$.
Let $M$ be a square $(0,1)$-matrix of order $n$ such that the first $m$ rows are all-one vector and the other $n-m$ rows are all-zero vector.
Then it is easy to check that the $p$-row graph of $M$ is isomorphic to $G$ for each $p\in[n]$.
\end{proof}

\begin{Cor}\label{cor:union_of_complete}
If $G$ is a disjoint union of complete graphs with $n$ vertices, then $\Upsilon(G)\supset[n-1]$.
\end{Cor}
\begin{proof}
Let $X_1,\ldots,X_k$ be the complete components of $G$.
Then $k\leq n$.
Now fix $i\in [n-1]$.
Then $k\leq {n \choose i}$.
Let $S_1,\ldots,S_k$ be distinct $i$-subsets of $[n]$.
Let $M$ be a matrix of order $n$ satisfying $\Lambda_M(v)=[n]\setminus S_j$ for each vertex $v$ in $X_j$ for each $j\in [k]$.
Then it is easy to check that $G$ is the $(n-i)$-row graph of $M$ and so $n-i\in\Upsilon(G)$ by Theorem~\ref{thm:rowgraph}.
\end{proof}

\begin{Ex}
The competition-realizer may be empty for some graph. For example,
for the complete bipartite graph $K_{3,3}$, $\Upsilon(K_{3,3})=\emptyset$.
To see why, we note that the number of vertices of $K_{3,3}$ is $6$ and $\theta_e(K_{3,3})=9$.
Therefore $1\notin \Upsilon(K_{3,3})$ by Theorem~\ref{thm:pECC}.
By Proposition~\ref{prop:lambda}, $5\notin \Upsilon(K_{3,3})$ and $6\notin \Upsilon(K_{3,3})$.
Suppose that $K_{3,3}$ is a $p$-competition graph for some $p \in \{2,3,4\}$.
Then $G$ is isomorphic to the $p$-row graph of a square $(0,1)$-matrix $M_p$ of order $6$ by Theorem~\ref{thm:rowgraph}.

Consider the case $p=4$.
Then each row of $M_4$ contains at least five $1$s by Proposition~\ref{prop:lambda}.
This implies that any two rows of $M_4$ have at least four common $1$s and so $G$ is isomorphic to $K_6$, which is a contradiction.
Thus $4\notin \Upsilon(K_{3,3})$.

Now consider the case $p=3$.
Then each row of $M_3$ contains at least four $1$s by Proposition~\ref{prop:lambda}.
This implies that any two rows of $M_3$ have at least two common $1$s.
If there is a row containing at least five $1$s, then it shares at least three common $1$s with each of the other vertices, which is impossible.
Thus each row of $M_3$ contains exactly four $1$s.
Since $K_{3,3}$ has a partite set of size $3$, we may assume that $M_3$ contains the following submatrix by permuting columns, if necessary:
\[
\begin{bmatrix}
  1 & 1 & 1 & 1 & 0 & 0 \\
  0 & 0 & 1 & 1 & 1 & 1 \\
  1 & 1 & 0 & 0 & 1 & 1
\end{bmatrix}.
\]
Now we take a vertex $u$ in the other partite set.
Then $u$ is adjacent to the vertex corresponding to each row of the above submatrix.
To have $u$ and the vertex corresponding to the first row of the above submatrix be adjacent, $\Lambda_{M_3}(u)\cap \{1,2,3,4\}$ is one of $\{1,2,3\}, \{2,3,4\}, \{1,2,4\}$, $\{1,3,4\}$, and $\{1,2,3,4\}$.
Then, in case of $\{1,2,3,4\}$, $u$ is not adjacent to the vertices corresponding to the second row and third row;
in case of $\{1,2,3\}$ or $\{1,2,4\}$, $u$ is not adjacent to the vertex corresponding to the second row; in case of $\{1,3,4\}$ or $\{2,3,4\}$, $u$ is not adjacent to the vertex corresponding to the third row.
Therefore we reach a contradiction.
Thus $3\notin \Upsilon(K_{3,3})$.

Now consider the case $p=2$.
Then each row of $M_2$ contains at least three $1$s by Proposition~\ref{prop:lambda}.
Suppose that there is a row ${\bf r}_1$ containing at least four $1$s.
Let $v_1$ be the vertex corresponding to ${\bf r}_1$ and $v_2$ and $v_3$ be the other vertices in the partite set to which $v_1$ belongs.
We may assume that $\Lambda_{M_2}(v_1)\supset \{1,2,3,4\}$.
Since $v_1$ and $v_2$ are not adjacent, ${\bf r}_1$ has exactly four $1$s and we may assume that the row ${\bf r}_2$ corresponding to $v_2$ has $1$ in the fourth component through the sixth component.
Then the row corresponding to $v_3$ must share at least two $1$s with ${\bf r}_1$ or ${\bf r}_2$ and we reach a contradiction.
Thus each row of $M_2$ contains exactly three $1$s.
If there are two vertices $w_1$ and $w_2$ in a partite set $W$ such that their corresponding rows do not share $1$s, then the row corresponding to the remaining vertex in $W$ must share at least two $1$s with one of the rows corresponding to $w_1$ and $w_2$, and we reach a contradiction.
Therefore the rows corresponding to two vertices in the same partite set share exactly one $1$.
Thus we may assume that $M_2$ contains the following submatrix by permuting columns, if necessary:
\[
\begin{bmatrix}
  1 & 1 & 1 & 0 & 0 & 0 \\
  0 & 0 & 1 & 1 & 1 & 0 \\
  1 & 0 & 0 & 0 & 1 & 1
\end{bmatrix}.
\]
Now we take a vertex $x$ in the other partite set $X$.
Then $x$ is adjacent to each of the vertices corresponding to the rows of the above submatrix.
Therefore $\Lambda_{M_2}(x)=\{1,3,5\}$.
Since $x$ is arbitrarily chosen, the rows corresponding to the other two vertices in $X$ also have $1$ in the first, the third, and the fifth component, which is impossible. Hence we have shown that $\Upsilon({K_{3,3}})=\emptyset$.
\end{Ex}

Now we introduce the notion of condensation of a graph.
Let $G$ be a graph with $n$ vertices.
Two vertices $u$ and $v$ of $G$ are said to be \emph{homogeneous}, denoted by $u \sim v$, if they have the same closed neighborhood.
Clearly $\sim$ is an equivalence relation on $V(G)$.
We denote the equivalence class containing a vertex $u$ of $G$ by $[u]$.
Then we define a new simple graph $G/\!\!\sim$ for $G$ by
\[
V(G / \!\!\sim)=\{[u] \mid u \in V(G)\} \quad  \text{ and } \quad  E(G / \!\!\sim)=\{[u][v] \mid  u, v \in V(G) \text{ and }   uv \in E(G)\}.
\]
See Figure~\ref{fig:contraction} for an illustration.
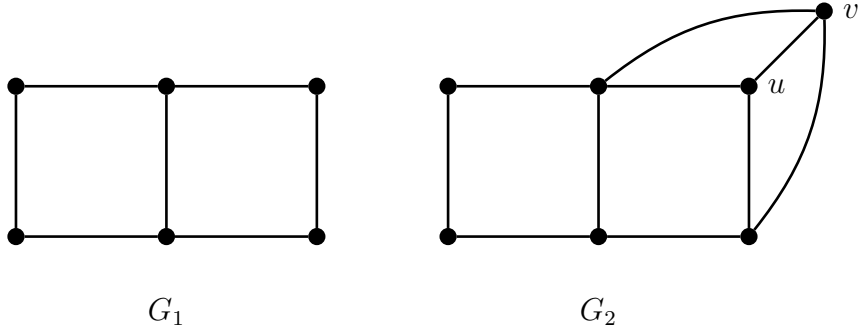
\begin{figure}
\centering{
\begin{tikzpicture}[every node/.style={transform shape}]
    \tikzset{mynode/.style={inner sep=2.3pt,fill,outer sep=0,circle}}
    \node (ghost) at (-3.4,0) {};
    \node[mynode] (1) at (-2,1) {};
    \node[mynode] (2) at (0,1) {};
    \node[mynode] (3) at (2,1) {};
    \node[mynode] (4) at (-2,-1) {};
    \node[mynode] (5) at (0,-1) {};
    \node[mynode] (6) at (2,-1) {};
    \node (name) at (0,-2) {$G_1$};

    \draw[line width=1pt] (2)--(3)--(6)--(5)--(4)--(1)--(2)--(5);
\end{tikzpicture}
\hspace{3em}
\begin{tikzpicture}[every node/.style={transform shape}]
    \tikzset{mynode/.style={inner sep=2.3pt,fill,outer sep=0,circle}}
    \node[mynode] (1) at (-2,1) {};
    \node[mynode] (2) at (0,1) {};
    \node[mynode] (3) at (2,1) {};
    \node[mynode] (4) at (-2,-1) {};
    \node[mynode] (5) at (0,-1) {};
    \node[mynode] (6) at (2,-1) {};
    \node[mynode] (7) at (3,2) {};
    \node (name) at (0,-2) {$G_2$};

    \draw[line width=1pt] (2)--(3)--(6)--(5)--(4)--(1)--(2)--(5);
    \draw[line width=1pt] (3)--(7);
    \path[line width=1pt] (2) edge [bend left=20] node {} (7);
    \path[line width=1pt] (6) edge [bend right=20] node {} (7);
    \draw[decorate,decoration={raise=3pt}] (3) node[right=3pt] {$u$};
    \draw[decorate,decoration={raise=3pt}] (7) node[right=3pt] {$v$};
\end{tikzpicture}}
\caption{$G_1\cong G_2/\!\!\sim$ (in $G_2$, $u\sim v$)}
\label{fig:contraction}
\end{figure}
We call $G/\!\!\sim$ the \emph{condensation} of $G$ for a graph $G$.

We note the following: \\
{\phantom{$\Leftrightarrow$} Two vertices $u$ and $v$ are adjacent in $G$
$\Leftrightarrow$ $v \in N_G[u]$
\\ $\Leftrightarrow$ $v \in N_G[u']$ for any $u' \in [u]$
$\Leftrightarrow$ $u' \in N_G[v]$ for any $u' \in [u]$
\\
$\Leftrightarrow$ $u' \in N_G[v']$ for any $u' \in [u]$ and $v' \in [v]$
\\
$\Leftrightarrow$
$u'$ and $v'$ are adjacent in $G$ for any $u' \in [u]$ and $v' \in [v]$.}
\\
Therefore $G / \!\!\sim$ is well-defined.

It is obvious that
\begin{itemize}
    \item[($\star$)] for each isolated vertex in $G$, its equivalence class is isolated in $G/\!\!\sim$.
\end{itemize}

The notion of $p$-row graph provides a way of getting information on the competition-realizer for a graph $G$ from the competition-realizer for a simpler graph $G/\!\!\sim$ as seen in the following results.
\begin{Prop}\label{prop:diam}
The condensation of $G$ has the same diameter as $G$ for a connected non-complete graph $G$.
\end{Prop}
\begin{proof}
Let $m$ be the diameter of $G$.
Since $G$ is not complete, $m\geq 2$.
Then there exists an induced $(u,v)$-path of length $m$ for some vertices $u$ and $v$ and there is no induced $(u,v)$-path of length $l$ for any $l < m$.
Thus, by the definition of condensation of a graph, there exists an induced $([u],[v])$-path of length $m$ and there is no induced $([u],[v])$-path of length $l$ for any $2 \leq l < m$ in $G/\!\!\sim$.
If there exists a $([u],[v])$-path of length $1$, then $u$ and $v$ are adjacent, which contradicts the choice of $u$ and $v$ so that $d_G(u,v)=m\geq 2$.
Therefore the diameter of $G/\!\!\sim$ is greater than or equal to $m$.

Let $m'$ be the diameter of $G/\!\!\sim$ for a nonnegative integer $m'$.
Since $G$ is non-complete, $m' \geq 1$.
Then there exists an induced $([x],[y])$-path of length $m'$ in $G/\!\!\sim$ and there is no induced $([x],[y])$-path of length $l$ for any $l < m'$ and some vertices $x$ and $y$ in $G$.
Therefore, by the definition of condensation of a graph, there is an induced $(x,y)$-path of length $m'$ and there is no induce $(x,y)$-path of length $l$ for any $l < m'$ in $G$.
Thus the diameter of $G/\!\!\sim$ is less than or equal to $m$ and this completes the proof.
\end{proof}

\begin{Prop}\label{prop:relative matrix}
A graph $G$ is a $p$-competition graph if and only if there exists a square matrix $M$ such that $G$ is a $p$-row graph of $M$ and the rows corresponding to two homogeneous vertices are identical.
\end{Prop}
\begin{proof}
The ``if'' part is obvious.
To show the ``only if'' part, suppose that a graph $G$ is a $p$-competition graph for some positive integer $p$.
Then, by Theorem~\ref{thm:rowgraph}, there exists a square matrix $M'$ such that $G$ is a $p$-row graph of $M'$.
If there are at least two homogeneous vertices, then we fix one row among the rows corresponding to them and replace the remaining rows with the fixed row.
We denote by $M$ the matrix obtained by repeatedly applying the above procedure.
It is easy to see that $G$ is a $p$-row graph of $M$.
\end{proof}

\begin{Prop}\label{prop:contraction}
Given a graph $G$ with $n$ vertices, suppose that $G / \!\!\sim$  is a $p$-row graph of a matrix $M$ satisfying the property that $M$ has $m$ columns for a positive  integer $m \le n$ and every row of $M$ has at least $p$ $1$s.
Then $\Upsilon(G)\supset\{p+i\mid i\in \{0,\ldots,n-m\}\}$.
\end{Prop}
\begin{proof}
Let $n_j$ be the size of equivalence class under $\sim$ corresponding to the $j$th row of $M$ for each $j$, $1\leq j \leq m$.
We replace the $j$th row of $M$ with $n_j$ copies of it to obtain the matrix $M^*$ which contains $M$ as a submatrix. We note that the size of $M^*$ is $n \times m$.
Suppose that $u$ and $v$ are two distinct vertices in $G$.
Then let $\mathbf{r}_u$ and $\mathbf{r}_v$ be the rows of $M^*$ corresponding to $u$ and $v$, respectively.
If $u$ and $v$ are not homogenous, then they belong to distinct equivalence classes under $\sim$ and the following are true:\vspace{-1ex}
\begin{itemize}
\item[\phantom{$\Leftrightarrow$}]Two vertices $u$ and $v$ are adjacent in $G$\vspace{-1ex}
\item[$\Leftrightarrow$]
 $[u]$ and $[v]$ are adjacent in $G/\!\!\sim$\vspace{-1ex}
\item[$\Leftrightarrow$] the row corresponding to $[u]$ and the row corresponding to $[v]$ have at least $p$ common $1$s in $M$\vspace{-1ex}
\item[$\Leftrightarrow$] $u$ and $v$ are adjacent in the $p$-row graph of $M^*$.\vspace{-1ex}
\end{itemize}
Suppose that $u$ and $v$ are homogenous. Then the rows ${\bf r}_u$ and ${\bf r}_v$ are identical.
By the hypothesis, every row of $M$ has at least $p$ $1$s.
Thus ${\bf r}_u$ and ${\bf r}_v$ have at least $p$ common $1$s and so $u$ and $v$ are adjacent in the $p$-row graph of $M^*$.
Hence $G$ is the $p$-row graph of $M^*$ and so, by Proposition~\ref{prop:prow}, $G$ is a $(p+i)$-competition graph, that is, $p+i \in \Upsilon(G)$ for any $i\in \{0,\ldots,n-m\}$.
\end{proof}

For a positive integer $p$ and the $p$-row graph $G$ of a matrix $M$, each non-isolated vertex in $G$ has at least $p$ $1$s in the row of $M$ corresponding to it and so the following corollary is immediately true by the above proposition.
\begin{Cor}\label{cor:contraction}
Given a graph $G$ with $n$ vertices, suppose that $G / \!\!\sim$ has no isolated vertices and is a $p$-row graph of a matrix $M$ having $m$ columns for a positive  integer $m \le n$.
Then $\Upsilon(G)\supset\{p+i \mid i\in \{0,\ldots,n-m\}\}$.
\end{Cor}

\begin{Cor}\label{cor:contraction_pComp}
Given a graph $G$ with $n$ vertices, suppose that
the number of non-isolated vertices in $G/\!\!\sim$ is $m$ for a positive integer $m \le n$.
Then {$$\Upsilon(G)\supset\{p+i\mid p\in \Upsilon(G/\!\!\sim),\ i\in \{0,\ldots,n-m\}\}.$$}
\end{Cor}
\begin{proof}
Suppose $p\in\Upsilon(G)$.
Then, by Theorem~\ref{thm:rowgraph}, $G/\!\!\sim$ is a $p$-row graph of a matrix $M$ having $m$ columns.
Thus by Corollary~\ref{cor:contraction}, 
$p+i \in \Upsilon(G)$ for any $i\in \{0,\ldots,n-m\}$.
\end{proof}


\begin{Rem}
Even if $G$ is a $p$-competition graph, $G/\!\!\sim$ may not be a $(p-i)$-competition graph for some $i\in [n-m]\cup\{0\}$ where $n=|V(G)|$ and $m=|V(G/\!\!\sim)|$.
For example, the graph $G_2$ in Figure~\ref{fig:contraction} is a $2$-competition graph.
To see why, we note that $G_1$ is the $2$-row graph of the matrix
\[
\begin{bmatrix}
  1 & 1 & 0 & 0 & 1 & 0 \\
  0 & 1 & 0 & 0 & 1 & 1 \\
  0 & 1 & 0 & 1 & 0 & 1 \\
  1 & 0 & 1 & 0 & 1 & 0 \\
  0 & 0 & 1 & 0 & 1 & 1 \\
  0 & 0 & 1 & 1 & 0 & 1
\end{bmatrix},
\]
so $G_1$ is a $2$-competition graph by Theorem~\ref{thm:rowgraph}.
Since $G_1$ is isomorphic to $G_2 /\!\!\sim$ and has no isolated vertices, $G_2$ is a $2$-competition graph by Corollary~\ref{cor:contraction_pComp}.
Yet, $G_1$, which is isomorphic to $G_2 /\!\!\sim$, is not a $1$-competition graph by Theorem~\ref{thm:ECC} since $|V(G_1)|=6 < 7=|E(G_1)|=\theta_e(G_1)$.
\end{Rem}

A {\it union} $G \cup H$ of two graphs $G$ and $H$ is the graph having its vertex set $V(G) \cup V(H)$ and edge set $E(G) \cup E(H)$.
In this paper, the union of $G$ and $H$ means their disjoint union which has an additional condition $V(G) \cap V(H) = \emptyset$.
A {\it join} $G \vee H$ of two vertex-disjoint graphs $G$ and $H$ is the graph having its vertex set $V(G) \cup V(H)$ and edge set $E(G)\cup E(H)\cup\{uv\mid u\in V(G), v\in V(H)\}$.

For a positive integer $n$, a nonnegative integer $k\leq n$, and the power set $\PPP([n])$ of $[n]$, we denote by ${\Psi}_{n,k}$ the simple graph with the vertex set $\PPP([n])$ and the edge set
\[\{ST\mid S,T\subset[n], |S\cup T| \leq k\}.\]

\begin{Thm}
Let $G$ be a connected graph with $n$ vertices and $k$ be a nonnegative integer less than $n$.
Then $G$ is an $(n-k)$-competition graph if and only if $G/\!\!\sim$ is isomorphic to an induced subgraph of $\Psi_{n,k}$.
\end{Thm}
\begin{proof}
To show the ``only if'' part, suppose that $G$ is an $(n-k)$-competition graph.
Then, by Proposition~\ref{prop:relative matrix}, there exists a square matrix $M$ of order $n$ such that $G$ is an $(n-k)$-row graph of $M$ and the rows corresponding to two homogeneous vertices are identical.
Let $M'$ be a submatrix of $M$ obtained by
taking all the distinct rows of $M$.
Then obviously $G/\!\!\sim$ is an $(n-k)$-row graph of $M'$.
Therefore two distinct vertices $[x]$ and $[y]$ are adjacent in $G/\!\!\sim$ if and only if $|\Lambda_{M'}([x])\cap \Lambda_{M'}([y])| \geq n-k$ if and only if $|([n]\setminus\Lambda_{M'}([x]))\cup ([n]\setminus\Lambda_{M'}([y]))| \leq k$ if and only if $[n]\setminus\Lambda_{M'}([x])$ and $[n]\setminus\Lambda_{M'}([y])$ are adjacent in $\Psi_{n,k}$.
Hence we have shown that $G/\!\!\sim$ is isomorphic to an induced subgraph of $\Psi_{n,k}$.

To show the ``if'' part, suppose that $G/\!\!\sim$ is isomorphic to an induced subgraph of $\Psi_{n,k}$.
Then each vertex $[v]$ of $G/\!\!\sim$ is assigned a subset $S_v$ of $[n]$ so that two distinct vertices $[v]$ and $[w]$ are adjacent in $G/\!\!\sim$ if and only if $|S_v \cup S_w|\leq k$.
If $|V(G/\!\!\sim)|=1$, then $G$ is complete and so $\Upsilon(G)=[n]$ by Proposition~\ref{prop:complete or empty}.
Thus, if $|V(G/\!\!\sim)|=1$, then $n-k \in \Upsilon(G)$.
Now suppose that $|V(G/\!\!\sim)| \geq 2$.
Since $G$ is connected by the hypothesis, $G/\!\!\sim$ is connected and so every vertex $[v]$ in $G/\!\!\sim$ is adjacent to a vertex, which implies $|S_v| \leq k$.
We denote by $M''$ the matrix with $n$ columns and with each row corresponds to a vertex of $G/\!\!\sim$ in such a way that $[n]\setminus\Lambda_{M''}([v])=S_v$.
Then it is easy to see that $G/\!\!\sim$ is an $(n-k)$-row graph of $M''$.
By Corollary~\ref{cor:contraction_pComp}, we can conclude that $n-k\in\Upsilon(G)$.
\end{proof}

\begin{Lem}\label{lem:induced star}
The star graph $K_{1,n}$ is an $n$-competition graph for a positive integer $n$.
\end{Lem}
\begin{proof}
It is obvious that $K_{1,n}$ is the $n$-row graph of the following square matrix of order $n+1$:
\[M=
\begin{bmatrix}
  1 & 1 & 1 & \cdots & 1 & 1 & 1 \\
  0 & 1 & 1 & \cdots & 1 & 1 & 1 \\
  1 & 0 & 1 & \cdots & 1 & 1 & 1 \\
  \vdots & \vdots & \vdots & \ddots & \vdots & \vdots & \vdots \\
  1 & 1 & 1 & \cdots & 0 & 1 & 1 \\
  1 & 1 & 1 & \cdots & 1 & 0 & 1
\end{bmatrix}.
\]
\end{proof}

The following proposition characterizes a graph $G$ with $n$ vertices and $n-1\in\Upsilon(G)$.

\begin{Prop}\label{prop:n-1}
Let $G$ be a graph with $n$ vertices. Then $G$ is an $(n-1)$-competition graph if and only if $G \cong (K_{n_0}\vee(K_{n_1}\cup\cdots\cup K_{n_k}))\cup \III_m$ for some nonnegative integers $k,n_0,n_1,\ldots,n_k$, and $m$ satisfying $m + \sum_{i=0}^k n_i = n$.
\end{Prop}
\begin{proof}
We show the ``if'' part.
If $G$ is empty, then, by Proposition~\ref{prop:complete or empty}, $G$ is an $(n-1)$-competition graph.
Now suppose that $G$ is a nonempty graph.
Let $G'$ be the subgraph of $G$ resulting from deleting all the isolated vertices in $G$.
It is easy to check that $G' /\!\!\sim$ is an empty graph or a star graph by the hypothesis.
Suppose that $|V(G'/\!\!\sim)|=1$.
Then $G'$ is a complete graph and so $\Upsilon(G')=\{1,\ldots,|V(G')|\}$.
By Proposition~\ref{prop:add isolated}, $\Upsilon(G)=[n]$ and so $G$ is an $(n-1)$-competition graph.
Now suppose that $|V(G'/\!\!\sim)| \geq 2$.
If $G'/\!\!\sim$ is an empty graph, then $G'$ is a disjoint union of complete graphs and so,
by Corollary~\ref{cor:union_of_complete} and Proposition~\ref{prop:add isolated}, $G$ is an $(n-1)$-competition graph.
Suppose that $G'/\!\!\sim$ is not an empty graph.
Then, by Proposition~\ref{prop:complete or empty} and Lemma~\ref{lem:induced star}, $|V(G'/\!\!\sim)|-1\in\Upsilon(G'/\!\!\sim)$.
Then $|V(G')|-1\in\Upsilon(G')$ by Corollary~\ref{cor:contraction_pComp}.
By Proposition~\ref{prop:add isolated}, $G$ is an $(n-1)$-competition graph.

Now we show the ``only if'' part.
Suppose that $G$ is an $(n-1)$-competition graph.
Then $G$ is isomorphic to the $(n-1)$-row graph of a matrix $M$ of order $n$ by Theorem~\ref{thm:rowgraph}.
Suppose that there exists a row, say $\mathbf{r}$, of $M$ such that $\mathbf{r}$ contains at most $n-2$ $1$s.
Then the vertex in $G$ corresponding to $\mathbf{r}$ is an isolated vertex,
so $G$ is still the $(n-1)$-row graph of the matrix resulting from replacing $\mathbf{r}$ with all-zero row.
Therefore we may assume that $M$ contains the rows of exactly three types:
\begin{itemize}
  \item[1.] the row with $n$ $1$s;
  \item[2.] the row with $n-1$ $1$s;
  \item[3.] the row with $0$ $1$s.
\end{itemize}
Let $V_1$, $V_2$, and $V_3$ be the sets of vertices corresponding to rows of Type i for each $i=1,2,3$.
Then the vertex set of $G$ is partitioned into $V_1$, $V_2$, and $V_3$.
Obviously, each vertex in $V_1$ is adjacent to each vertex of $G$ not in $V_3$ and each vertex in $V_3$ is isolated.
We note that two rows of Type 2 are identical if and only if their corresponding vertices in $V_2$ are adjacent in $G$.
Therefore each vertex in $V_2$ is a simplicial vertex, that is, a vertex whose neighbors form a clique.
Let $n_0 = |V_1|$ and $m = |V_3|$.
If $V_2 = \emptyset$, then $G \cong K_{n_0} \cup \III_m$ by the above observation.
Now suppose that $V_2\neq\emptyset$.
Then the subgraph of $G$ induced by $V_2$ is a disjoint union of cliques.
Let $W_1,W_2,\ldots,W_k$ be the vertex sets of those cliques and let $|W_i|=n_i$ for $1 \le i \le k$. Then $m + \sum_{i=0}^k n_i = n$.
By the above observation again, $G \cong (K_{n_0}\vee(K_{n_1}\cup\cdots\cup K_{n_k}))\cup \III_m$.
\end{proof}

\begin{Thm}
Let $G$ be a graph with $n$ vertices.
Then $\Upsilon(G)=[n-1]$ if and only if $G\cong H\cup \III_m$ for some integer $m$, $0\leq m \leq n-3$ and some graph $H$ for which $H/\!\!\sim$ is an induced subgraph of a star graph and has more than one vertex.
\end{Thm}
\begin{proof}
To show the ``if'' part, suppose that $G\cong H\cup \III_m$ for some integer $m$, $0\leq m \leq n-3$ and some graph $H$ for which $H/\!\!\sim$ is an induced subgraph of a star graph $Q$ with more than one vertex.
We denote the number of vertices in $H/\!\!\sim$ by $t$.
Take $p\in[t-1]$.
We construct a square $(0,1)$-matrix $M$ of order $t$ in the following way.
If $H/\!\!\sim$ contains a center of $Q$, then the row of $M$ corresponding to it is the all-one vector.
The rows of $M$ corresponding to the vertices in $H/\!\!\sim$ which are not a center of $Q$ are mutually distinct, and the number of $1$s in each of them is $p$.
Such a matrix $M$ exists since ${t \choose p} \geq t$.
It is easy to check that $H/\!\!\sim$ is isomorphic to the $p$-row graph of $M$.
Thus $[t-1] \subset \Upsilon(H/\!\!\sim)$ by Theorem~\ref{thm:rowgraph}.
By Proposition~\ref{prop:contraction}, $[n-m-1] \subset \Upsilon(H)$.
Now, by Proposition~\ref{prop:add isolated}, $[n-1]\subset \Upsilon(G)$.
By Proposition~\ref{prop:complete}, $n\notin\Upsilon(G)$ and so $\Upsilon(G)=[n-1]$.

To show the ``only if" part, suppose that
$\Upsilon(G)=[n-1]$.
Then, by Proposition~\ref{prop:n-1}, $G \cong (K_{n_0}\vee(K_{n_1}\cup\cdots\cup K_{n_k}))\cup \III_m$ for some nonnegative integers $k,n_0,n_1,\ldots,n_k$, and $m$ satisfying $m + \sum_{i=0}^k n_i = n$.
If there is at most one nonzero integer among $n_1,n_2,\ldots,n_k$, then $G\cong K_{n-l_1} \cup \III_{l_1}$ for some integer $l_1$, $0\leq l_1 \leq n$ and so, by Corollary~\ref{cor:upsilon_n}, $\Upsilon(G)=[n]$, which is a contradiction.
Therefore there are at least two nonzero integers among $n_1,n_2,\ldots,n_k$ and $H:=K_{n_0}\vee (K_{n_1}\cup\cdots\cup K_{n_k})$ is an induced subgraph of $G$.
It is easy to check that $H/\!\!\sim$ is an induced subgraph of a star graph and has more than one vertex.
Finally we show that $m \leq n-3$.
Suppose that there are exactly two nonzero integers among $n_0,n_1,n_2,\ldots,n_k$.
If $n_0$ and $n_i$ are nonzero integers for some $i$, $1\leq i \leq k$, then $G\cong K_{n_0+n_i} \cup \III_{l_2}$ for some integer $l_2$, $0 \leq l_2 \leq n$ and we reach a contradiction as before.
If $n_i$ and $n_j$ are nonzero integers for some $i$ and $j$, $1\leq i < j \leq k$, then $G \cong K_{n_i} \cup K_{n_j} \cup \III_{l_3}$ for some integer $l_3$, $0 \leq l_3 \leq n$, which implies $n_i \geq 2$ and $n_j \geq 2$.
Therefore $m \leq n-4$ if there are exactly two nonzero integers among $n_0,n_1,n_2,\ldots,n_k$.
If there are at least three nonzero integers among $n_0,n_1,n_2,\ldots,n_k$, then it is obvious that $m\leq n-3$.
\end{proof}

A \emph{hole} of a graph is a cycle of length greater than or equal to $4$ which is an induced subgraph of the graph.
A graph without holes is said to be \emph{chordal}.

For $p=n$ or $n-1$, a $p$-competition graph is a chordal graph by Propositions~\ref{prop:complete} and \ref{prop:n-1}. As a matter of fact, an $(n-2)$-competition graph is also chordal by Corollary~\ref{cor:hole}.

An \emph{induced path} of a graph means a path as an induced subgraph of the graph.

\begin{Thm}\label{prop:n-3}
If a graph $G$ with $n$ vertices  contains two internally disjoint induced paths of length $2$ (whose internal vertices are nonadjacent), then $\Upsilon(G) \subset [n-3]$.
\end{Thm}
\begin{proof}
Let $G$ be a graph with $n$ vertices containing two internally disjoint induced paths $uvw$ and $xyz$ of length $2$ with $v$ and $y$ nonadjacent (see Figure~\ref{fig:chordal}).
Then, by Propositions~\ref{prop:complete} and  \ref{prop:n-1}, $\Upsilon(G) \subset [n-2]$.
Suppose, to the contrary, that $n-2 \in \Upsilon(G)$.
By Theorem~\ref{thm:rowgraph}, $G$ is isomorphic to the $(n-2)$-row graph of a square matrix $M$ of order $n$.
If $|\Lambda_M(v)|\geq n-1$ and $|\Lambda_M(y)|\geq n-1$, then $|\Lambda_M(v)\cap \Lambda_M(y)|\geq n-2$ and so $v$ and $y$ are adjacent, which is impossible.
Thus $|\Lambda_M(v)|\leq n-2$ or $|\Lambda_M(y)|\leq n-2$.
Without loss of generality, we may assume that $|\Lambda_M(v)|\leq n-2$.
Since $v$ is non-isolated, $|\Lambda_M(v)|=n-2$.
Since $u$ and $v$ (resp.\ $w$ and $v$) are adjacent, $|\Lambda_M(u)\cap \Lambda_M(v)|\geq n-2$ (resp.\ $|\Lambda_M(w)\cap \Lambda_M(v)|\geq n-2$).
Since $|\Lambda_M(v)|=n-2$, $\Lambda_M(v)\subset \Lambda_M(u)$ and $\Lambda_M(v)\subset \Lambda_M(w)$.
Therefore $\Lambda_M(v)\subset\Lambda_M(u)\cap \Lambda_M(w)$ and so $|\Lambda_M(u)\cap \Lambda_M(w)|\geq n-2$.
Then $u$ and $w$ are adjacent in $G$ and we reach a contradiction.
\end{proof}
\begin{figure}
\centering
\begin{tikzpicture}[every node/.style={transform shape},scale=.8]
    \tikzset{mynode/.style={inner sep=2.3pt,fill,outer sep=0,circle}}

    \node[mynode] (x) at (-2,-1) {};
    \node[mynode] (y) at (0,-1) {};
    \node[mynode] (z) at (2,-1) {};
    \node[mynode] (u) at (-2,1) {};
    \node[mynode] (v) at (0,1) {};
    \node[mynode] (w) at (2,1) {};

    \draw (x) -- (y) -- (z);
    \draw (u) -- (v) -- (w);
    \draw[dashed] (v) -- (y);
    \draw[dashed] (x) edge[bend right] (z);
    \draw[dashed] (u) edge[bend left] (w);

    \draw[decorate,decoration={raise=3pt}] (u) node[left=3pt] {$u$};
    \draw[decorate,decoration={raise=3pt}] (v) node[above=3pt] {$v$};
    \draw[decorate,decoration={raise=3pt}] (w) node[right=3pt] {$w$};
    \draw[decorate,decoration={raise=3pt}] (x) node[left=3pt] {$x$};
    \draw[decorate,decoration={raise=3pt}] (y) node[below=3pt] {$y$};
    \draw[decorate,decoration={raise=3pt}] (z) node[right=3pt] {$z$};
\end{tikzpicture}
\caption{Paths $uvw$ and $xyz$ given in the proof of Theorem~\ref{prop:n-3}. The dotted line between two vertices means that they are not adjacent.}\label{fig:chordal}
\end{figure}
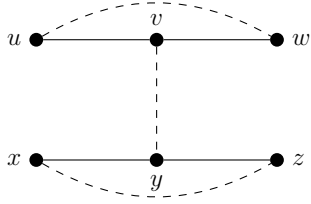
Since a graph with the diameter at least four has two internally disjoint induced paths of length $2$, the following corollary is immediately true.
\begin{Cor}\label{cor:n-3}
If the diameter of a graph $G$ is at least four, then $\Upsilon(G)\subset [n-3]$.
\end{Cor}

\begin{Cor}\label{cor:hole}
If $\Upsilon(G)\not\subset [n-3]$ for a graph $G$ with $n\geq 4$ vertices, then $G$ is chordal and has no induced path of length $4$.
\end{Cor}
\begin{proof}
We will show that the contrapositive of the statement is true.
If a $p$-competition graph with $n$ vertices is non-chordal or has an induced path of length $4$ for integers $n\geq 4$ and $p\in[n]$, then it contains two internally disjoint induced paths of length $2$ whose internal vertices are nonadjacent and the statement is true by Theorem~\ref{prop:n-3}.
\end{proof}

By Corollary~\ref{cor:K_n}, it is trivially true that if a connected graph $G$ with $n$ vertices is an $n$-competition graph, then the diameter of $G$ is $1$. By Proposition~\ref{prop:n-1}, the diameter of a connected $(n-1)$-competition graph which has $n$ vertices is at most $2$. The diameter of a connected $(n-2)$-competition graph which has $n$ vertices is at most $3$ by Corollary~\ref{cor:hole}. However, interestingly, the diameter of a connected $(n-3)$-competition graph with $n$ vertices can be arbitrarily large, which will be shown by Proposition~\ref{lem:path}.

\begin{Prop}\label{prop:theta}
For a graph $G$ with $\theta_e(G) \le |V(G)|$, $[|V(G)|-\theta_e(G)+1]\subset \Upsilon(G)$.
\end{Prop}
\begin{proof}
Let $|V(G)|=n$ and $V(G)=\{v_1, v_2, \ldots, v_n\}$.
There is an edge clique cover $\CCC:=\{C_1, C_2, \ldots, C_{\theta_e(G)} \}$ of $G$ as $\theta_e(G)$ is the edge clique number of $G$.
We define an $n \times \theta_e(G)$ matrix $M=(m_{ij})$ as follows:
\[
m_{ij} = %
\begin{cases}
    1& \mbox{if } v_i \in C_j   \\
    0& \mbox{if } v_i \notin C_j.
\end{cases}
\]
Then $G$ is isomorphic to the $1$-row graph of $M$.
Therefore the statement is true by Proposition~\ref{prop:prow}.
%
\end{proof}
It is a well-known fact that any graph $G$ can be made into the competition graph of an acyclic digraph as long as it is allowed to add new isolated vertices to $G$.
The smallest among such numbers is called the \emph{competition number} of $G$ and denoted by $k(G)$.
Opsut~\cite{opsut1982computation} showed that
\begin{equation}\label{eqn:competition}
k(G) \ge \theta_e(G)-|V(G)|+2.
\end{equation}

\begin{Cor}\label{cor:omega}
Let $G$ be a graph with $\omega$ components.
If each component of $G$ has competition number one, then $[\omega+1]\subset\Upsilon(G)$.
\end{Cor}
\begin{proof}
Let $G_1$, $G_2$, $\ldots$, $G_\omega$ be the components of $G$.
Then, by \eqref{eqn:competition}, $|V(G_i)|-\theta_e(G_i) \ge 2-k(G_i) = 1$ for any $1 \le i \le \omega$.
Since $|V(G)| = \sum_{i=1}^{\omega} |V(G_i)|$ and $\theta_e(G) = \sum_{i=1}^{\omega} \theta_e(G_i)$,  $|V(G)| - \theta_e(G)+1 \ge \omega+1$.
Thus, by Proposition~\ref{prop:theta}, the corollary statement is true.
\end{proof}

It is known that a chordal graph and a forest both have the competition number at most $1$.
Since any graph without isolated vertex has competition number at least $1$, the following corollaries immediately follow from Corollary~\ref{cor:omega}.
\begin{Cor}
For a chordal graph $G$ having $\omega$ components none of which is trivial, $[\omega+1]\subset \Upsilon(G)$.
\end{Cor}

\begin{Cor}\label{cor:forest}
For a forest $G$ having $\omega$ components none of which is trivial, $[\omega+1]\subset \Upsilon(G)$.
\end{Cor}

\section{Competition-realizers for trees}

{In this section, we study $p$-competition trees. Especially, we completely characterize the competition-realizers for caterpillars.}

Let $G$ be a $p$-competition graph. Then $G$ is isomorphic to the $p$-row graph of a matrix $M=(m_{ij})$.
If $|\Lambda_M(v)| \le p-1$, then $v$ is an isolated vertex in $G$, and so $G$ is still the $p$-row graph of the resulting matrix even if the row corresponding to $v$ is replaced by the row with $p-1$ $1$s. Thus we may conclude as follows:
\begin{itemize}
\item[$(\S)$] If a $p$-competition graph $G$ is isomorphic to the $p$-row graph of a matrix $M$, then we may assume that $|\Lambda_M(v)|\geq p-1$ for each vertex $v$ in $G$.
\end{itemize}

\emph{Adding a pendant vertex $v$} to a graph $G$ means obtaining a graph $G'$ such that $v\notin V(G)$, $V(G')=V(G)\cup\{v\}$, and $E(G')=E(G)\cup\{ vu\}$ for a vertex $u$ in $G$.

\begin{Thm}\label{thm:addVrtx}
If $G$ is a graph obtained from a $p$-competition graph by adding a pendant vertex, then $p\in \Upsilon(G)$.
\end{Thm}
\begin{proof}
Let $G$ be a graph obtained from a $p$-competition graph $G'$ with $n$ vertices by adding a pendant vertex $u$ at vertex $v$ in $G'$.
Since $G'$ is a $p$-competition graph, $G'$ is isomorphic to the $p$-row graph of a square matrix $M'=(m'_{ij})$ of order $n$.
By ($\S$), we may assume that $|\Lambda_{M'}(v)| \ge p-1$.
Without loss of generality, we may assume that the row corresponding to $v$ is located at the bottom of $M'$ and $\Lambda_{M'}(v)=\{1,2,\ldots,|\Lambda_{M'}(v)|\}$.

Now we define a matrix $M=(m_{ij})$ of order $n+1$ by
\begin{center}
\begin{tikzpicture}[
    style1/.style={
        matrix of math nodes,
        every node/.append style = {
            text width=#1,
            align=center,
            minimum height=3ex
            },
        nodes in empty cells,
        left delimiter = [,
        right delimiter = ],
        column 1/.style={
            nodes={text width=3.3*#1}
        },
        row 1/.style={
            nodes={text height=9ex}
        },
        column 2/.style={
            nodes={text width=.3*#1}
        }
    }
]
    \matrix[style1 = 3ex] (add1) { & & \\ & & \\ & & \\};
    \draw[dashed] (add1-1-2.north) -- (add1-3-2.south);
    \draw[dashed] (add1-2-1.west) -- (add1-2-3.east);

    \node at (add1-1-1) {$M'$};
    \node[align=center] at (add1-1-3) {$0$ \\ $\vdots$ \\ $0$ \\ $1$};
    \node at (add1-3-1) {$1\cdots1\;0\cdots0$};
    \node at (add1-3-3) {$1$};

    \draw[decorate, decoration={raise=12pt}] (add1-3-3.east) node[right=12pt] {$u$};
    \draw[decorate, decoration={raise=12pt}] ([yshift=-23pt]add1-1-3.east) node[right=12pt] {$v$};
    \draw[decorate, decoration={mirror,raise=12pt}] ([yshift=-17pt]add1-1-1.west) node[left=12pt] {$M=$};
    \draw[decorate, decoration={mirror,brace,raise=5pt}] (add1-3-1.south west) -- node[font=\scriptsize,below=5pt] {$p-1$} (add1-3-1.south);
    \draw[decorate, decoration={mirror,brace,raise=5pt}] (add1-3-1.south) -- node[font=\scriptsize,below=5pt] {$n-p+1$} (add1-3-2.south);
    \draw[decorate, decoration={mirror,brace,raise=5pt}] (add1-3-2.south) -- node[font=\scriptsize, below=5pt] {$1$} (add1-3-3.south east);
\end{tikzpicture}
\end{center}
It is easy to check that $G$ is the $p$-row graph of $M$. By Theorem~\ref{thm:rowgraph}, $G$ is a $p$-competition graph.
\end{proof}

Kim \textit{et al.}~\cite{kim2009cycles} specified the length of a cycle which is a $p$-competition graph in terms of $p$.
\begin{Thm}[\!\!\cite{kim2009cycles}]\label{thm:cycle}
Let $C_n$ be a cycle with $n$ vertices for a positive integer $n\geq 4$.
Then $\Upsilon(C_n)=[n-3]$.
\end{Thm}
\noindent In the proof of Theorem~\ref{thm:cycle}, $\mathcal{F}:=\{S_0,\ldots,S_{n-1}\}$, where, for each $i=0, \ldots, n-1$, $S_i:=\{v_i,v_{i+1},\ldots,v_{i+p}\}$, is given as a $p$-edge clique cover of $C_n=v_0v_1\ldots v_{n-1}v_0$ for which all the subscripts are reduced modulo $n$.
The following square matrix of order $n$ is obtained from $\mathcal{F}$ by \eqref{eqn:matrix}:
\begin{equation}\label{eqn:cyclematrix}
M_{p,n}:=\begin{bmatrix}
          1 & 0 & 0 & 0 &  \cdots & 0 & 0 & 1 & 1 & 1 & \cdots & 1  \\
          1 & 1 & 0 & 0 &  \cdots & 0 & 0 & 0 & 1 & 1 & \cdots & 1  \\
          1 & 1 & 1 & 0 &  \cdots & 0 & 0 & 0 & 0 & 1 & \cdots & 1  \\
          \vdots & \vdots & \vdots & \vdots &  \vdots & \vdots & \vdots & \vdots & \vdots & \vdots & \vdots & \vdots  \\
          1 & 1 & 1 & 1 &  \cdots & 1 & 1 & 0 & 0 & 0 & \cdots & 0  \\
          0 & 1 & 1 & 1 &  \cdots & 1 & 1 & 1 & 0 & 0 & \cdots & 0  \\
           \vdots & \vdots & \vdots & \vdots &  \vdots & \vdots & \vdots & \vdots & \vdots & \vdots & \vdots & \vdots \\
           0 & 0 & 0 & 0 &  \cdots & 0 & 1 & 1 & 1 & 1 & \cdots & 1  \\
        \end{bmatrix},
\end{equation}
where $i$th row (resp.\ column) is corresponding to $v_{i-1}$ (resp.\ $S_{i-1}$) and the $(i+1)$st row (we identify the $(n+1)$st row with the first row) is obtained by cyclically shifting the $i$th row by $1$ to the right for each $i$, $1\le i \le n$.
Therefore the $p$-row graph of $M_{p,n}$ is isomorphic to $C_n$ and the following proposition is immediately true.
\begin{Lem}\label{prop:sets}
    Let $n$ be an integer greater than or equal to $4$ and $p$ be a positive integer less than or equal to $n-3$. Then, for the matrix $M_{p,n}$ given in the equation~\eqref{eqn:cyclematrix}, the following are true:
    \begin{itemize}
      \item[(1)] the $k$th row contains exactly $p+1$ $1$s for each integer $k$, $1 \le k \le n$;
      \item[(2)] the $k$th row and the $(k+1)$st row have common $1$s in exactly $p$ columns for each integer $k$, $1 \le k \le n$ (we identify the $(n+1)$st row with the first row);
      \item[(3)] the $k$th row and the $l$th row have common $1$s in at most $p-1$ columns for integers $k$ and $l$ satisfying $1 \le k, l \le n$ and $2 \le |k-l|$.
    \end{itemize}
\end{Lem}

We denote the path graph with $n$ vertices by $P_n$.
\begin{Prop}\label{lem:path}
For an integer $n \ge 3$,
\[\Upsilon(P_n)=\begin{cases}
\{1,2\} &  \mbox{if } n \in \{3,4\}; \\
[n-3] & \mbox{if } n \ge 5.
\end{cases}\]
\end{Prop}
\begin{proof}
For $n \ge 5$, then $\Upsilon(P_n) \subset [n-3]$ by Corollary~\ref{cor:hole}.
By Corollary~\ref{cor:K_n}, $\Upsilon(P_3) \subset \{1,2\}$.
By Corollary~\ref{cor:K_n} and Proposition~\ref{prop:n-1}, $\Upsilon(P_4) \subset \{1,2\}$.

Now we show the other direction containment.
If $n= 3$ and $p\leq 2$, then $\{1,2\}\subset\Upsilon(P_n)$ by Corollary~\ref{cor:forest}.
It is easy to check that $P_4$ is isomorphic to the $2$-row graph of the following matrix:
\[M^*_{2,4}:=\begin{bmatrix}
  1 & 1 & 0 & 0 \\
  1 & 1 & 1 & 0 \\
  0 & 1 & 1 & 1 \\
  0 & 0 & 1 & 1
\end{bmatrix}.\]
Thus $2\in \Upsilon(P_4)$.
Now suppose that $n\ge 4$ and $p\le n-3$.
In $M_{p,n}$ given in \eqref{eqn:cyclematrix}, we replace $1$ in the $(1,n-p+1)$-entry with $0$ to obtain a square matrix $M^*_{p,n}$ of order $n$.
Let $G'$ be the $p$-row graph of $M^*_{p,n}$.
Then the first row and the second row of $M^*_{p,n}$ still share $p$ $1$s.
Yet the first row and the $n$th row of $M^*_{p,n}$ share only $p-1$ $1$s.
Thus, by (2) and (3) of Lemma~\ref{prop:sets}, $G'$ is isomorphic to a path graph with $n$ vertices.
Hence $[n-3]\subset \Upsilon(P_n)$ and this completes the proof.
\end{proof}

\begin{Prop}\label{prop:tree_diam}
Let $T$ be a tree with the diameter $m$ for some integer $m$, $m \geq 3$. Then $\Upsilon(T)\supset[m-2]$.
\end{Prop}
\begin{proof}
Take $p\in [m-2]$. Since the diameter of $T$ is $m$, there exists an induced path of length $m$. Since $m \ge p+2$, we may take a section $P$ of this path which has length $p+2$. Then $P$ is a $p$-competition graph by Proposition~\ref{lem:path}. Since $T$ can be obtained from $P$ by adding pendant vertices sequentially, $G$ is a $p$-competition graph by Theorem~\ref{thm:addVrtx}.
\end{proof}

\begin{Cor}
Given a graph $G$ with $n$ vertices and diameter $m$, if $G/\!\!\sim$ is a tree with $n'$ vertices, then $\Upsilon(G)\supset [n-n'+m-2]$.
\end{Cor}
\begin{proof}
Suppose that $G/\!\!\sim$ is a tree with $n'$ vertices.
Then, by Proposition~\ref{prop:diam}, $G/\!\!\sim$ has diameter $m$.
Thus, by Proposition~\ref{prop:tree_diam}, $\Upsilon(G/\!\!\sim)\supset[m-2]$.
By Corollary~\ref{cor:contraction_pComp},
\[\Upsilon(G)\supset\{p+i\mid p\in [m-2] \text{ and }i\in[n-n']\cup\{0\}\}=[n-n'+m-2].\]
\end{proof}
A \emph{caterpillar} is a tree with at least $3$ vertices the removal of whose pendant vertices produces a path.
A \emph{spine} of a caterpillar is the longest path of the caterpillar.
In the following, for a caterpillar $T$ with $n$ vertices, we shall find all the positive integers $p$ such that $T$ is a $p$-competition graph in terms of $n$. To do so, we need the following lemmas.

\begin{Lem} \label{lem:caterpillar}
Let $n$ and $t$ be positive integers such that either $(n,t)=(4,2)$ or $n\geq 4$ and $t\leq n-3$.
Then, for any nonnegative integer $k$ and a path graph $P$ of length $n-1$, a caterpillar $T$ obtained  by adding $k$ new vertices to $P$ in such a way that the added vertices are pendent vertices of $T$ adjacent to interior vertices of $P$ is a $(t+k)$-competition graph.
\end{Lem}
\begin{proof}
Since either $(n,t)=(4,2)$ or $n \ge 4$ and $t \le n-3$, $P$ is a $t$-competition graph by Proposition~\ref{lem:path}, which is actually the $t$-row graph of $M^*_{t,n}$ where $M^*_{t,n}$ is the matrix defined in the proof of the same lemma.
Thus the statement is true for $k=0$.
Now we assume that $k$ is a positive integer.
Let $P=x_1x_2\cdots x_n$ and  $y_1,\ldots,y_k$ be the vertices added to $P$ as described in the theorem statement.
There exists a map $\phi:[k]\to[n]$ such that $x_{\phi(i)}$ is a vertex on $P$ adjacent to $y_i$ for each $i$, $1 \leq i \leq k$.
By the way that $y_1,\ldots,y_k$ were added, $\phi$ is well-defined.
We define a $k \times n$ $(0,1)$-matrix $A$ so that the $i$th row of $A$ is the same as the row of $M^*_{t,n}$ corresponding to $x_{\phi(i)}$ for each $i$, $1 \leq i \leq k$.

Now we consider the matrix $M$ defined as follows:
\begin{center}
\begin{tikzpicture}[style1/.style={matrix of math nodes, every node/.append style = {text width=#1,align=center,minimum height=5.5ex}, nodes in empty cells, left delimiter = [, right delimiter = ],}]%
    \matrix[style1 = 1.45cm] (cater) { & & \\
      & & \\
      & & \\ };
    \draw[dashed] (cater-2-1.west) -- (cater-2-3.east);%
    \draw[dashed] (cater-1-2.north) -- (cater-3-2.south);%
    \node[] at ([xshift=5pt,yshift=-3pt]cater-1-1) {$M_{t,n}^*$};%
    \node[] at ([xshift=-5pt,yshift=-3pt]cater-1-3) {$J_{n,k}$};%
    \node[] at ([xshift=5pt,yshift=3pt]cater-3-1) {$A$};%
    \node[] at ([xshift=-5pt,yshift=3pt]cater-3-3) {$J_{k,k}-I_k$};%
    \draw[decorate, decoration={brace,raise=5pt}] ([xshift=-5pt]cater-1-1.north west) -- node[font=\footnotesize,above=5pt] {$n$} (cater-1-2.north);%
    \draw[decorate, decoration={brace,raise=5pt}] (cater-1-2.north) -- node[font=\footnotesize,above=5pt] {$k$} (cater-1-3.north east);%


    \draw[decorate,decoration={raise=12pt}](cater-2-1.west) node[left=12pt] {$M:=$};
    \draw[decorate,decoration={mirror, raise=12pt}]([yshift=-5pt]cater-2-3.east) node[right=12pt] {$.$};
    \end{tikzpicture}
\end{center}
For an example, see Figure~\ref{fig:caterpillar}.
Then, for each $i$ and $j$, $1\leq i < j \leq k$,
\begin{itemize}
  \item[(M-1)] if $y_i$ and $y_j$ are adjacent to the same vertex on $P$, then the $i$th row and the $j$th row of the $(2,1)$-block of $M$ are identical and have exactly $t+1$ common $1$s; otherwise, the $i$th row and the $j$th row of the $(2,1)$-block of $M$ have at most $t$ common $1$s;
  \item[(M-2)] if $y_i$ and $x_l$ are adjacent in $T$ for some $l\in[n]$, then the $l$th row of the $(1,1)$-block of $M$ and the $i$th row of the $(2,1)$-block of $M$ have exactly $t+1$ common $1$s; otherwise, the $l$th row of the $(1,1)$-block of $M$ and the $i$th row of the $(2,1)$-block of $M$ have at most $t$ common $1$s.
\end{itemize}

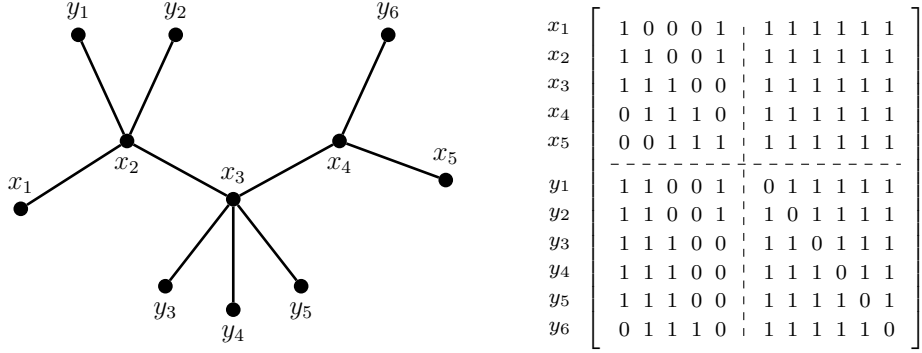
\begin{figure}
\begin{center}
\begin{tikzpicture}[scale=0.85, every node/.style={transform shape}]
    \tikzset{mynode/.style={inner sep=2.3pt,fill,outer sep=0,circle}}
    \node[mynode] (x3) at (0,0) {};
    \node[mynode] (x4) [above right = .75cm and 1.5cm of x3] {};
    \node[mynode] (x5) [below right = .45cm and 1.5cm of x4] {};
    \node[mynode] (x2) [above left = .75cm and 1.5cm of x3] {};
    \node[mynode] (x1) [below left = .9cm and 1.5cm of x2] {};
    \node[mynode] (y1) [above left = 1.5cm and .6cm of x2] {};
    \node[mynode] (z1) [above right = 1.5cm and .6cm of x2] {};
    \node[mynode] (y2) [below left = 1.2cm and .9cm of x3] {};
    \node[mynode] (z2) [below = 1.5cm of x3] {};
    \node[mynode] (z3) [below right = 1.2cm and .9cm of x3] {};
    \node[mynode] (y3) [above right = 1.5cm and .6cm of x4] {};

    \draw[line width=1pt] (x1) -- (x2) -- (x3) -- (x4) -- (x5);
    \draw[line width=1pt] (x2) -- (y1);
    \draw[line width=1pt] (x2) -- (z1);
    \draw[line width=1pt] (x3) -- (y2);
    \draw[line width=1pt] (x3) -- (z2);
    \draw[line width=1pt] (x3) -- (z3);
    \draw[line width=1pt] (x4) -- (y3);

    \draw[decorate,decoration={raise=3pt}] (x1) node[above=3pt] {$x_1$};
    \draw[decorate,decoration={raise=3pt}] (x2) node[below=3pt] {$x_2$};
    \draw[decorate,decoration={raise=3pt}] (x3) node[above=3pt] {$x_3$};
    \draw[decorate,decoration={raise=3pt}] (x4) node[below=3pt] {$x_4$};
    \draw[decorate,decoration={raise=3pt}] (x5) node[above=3pt] {$x_5$};
    \draw[decorate,decoration={raise=3pt}] (y1) node[above=3pt] {$y_1$};
    \draw[decorate,decoration={raise=3pt}] (z1) node[above=3pt] {$y_2$};
    \draw[decorate,decoration={raise=3pt}] (y2) node[below=3pt] {$y_3$};
    \draw[decorate,decoration={raise=3pt}] (z2) node[below=3pt] {$y_4$};
    \draw[decorate,decoration={raise=3pt}] (z3) node[below=3pt] {$y_5$};
    \draw[decorate,decoration={raise=3pt}] (y3) node[above=3pt] {$y_6$};
\end{tikzpicture}
\qquad
\begin{tikzpicture}[style1/.style={matrix of math nodes, every node/.append style = {text width=#1,align=center,minimum height=1ex}, nodes in empty cells, left delimiter = [, right delimiter = ],}]%
\scriptsize
    \matrix[style1 = 1ex] (cater) {%
    1 & 0 & 0 & 0 & 1 & & 1 & 1 & 1 & 1 & 1 & 1 \\%
    1 & 1 & 0 & 0 & 1 & & 1 & 1 & 1 & 1 & 1 & 1 \\%
    1 & 1 & 1 & 0 & 0 & & 1 & 1 & 1 & 1 & 1 & 1 \\%
    0 & 1 & 1 & 1 & 0 & & 1 & 1 & 1 & 1 & 1 & 1 \\%
    0 & 0 & 1 & 1 & 1 & & 1 & 1 & 1 & 1 & 1 & 1 \\%
     & & & & & & & & & & & \\%
    1 & 1 & 0 & 0 & 1 & & 0 & 1 & 1 & 1 & 1 & 1 \\%
    1 & 1 & 0 & 0 & 1 & & 1 & 0 & 1 & 1 & 1 & 1 \\%
    1 & 1 & 1 &0 & 0 & & 1 & 1 & 0 & 1 & 1 & 1 \\%
    1 & 1 & 1 &0 & 0 & & 1 & 1 & 1 & 0 & 1 & 1 \\%
    1 & 1 & 1 &0 & 0 & & 1 & 1 & 1 & 1 & 0 & 1 \\%
    0 & 1 & 1 & 1 & 0 & & 1 & 1 & 1 & 1 & 1 & 0 \\%
     };%
    \draw[dashed] (cater-6-1.west) -- (cater-6-12.east);%
    \draw[dashed] (cater-1-6.north) -- (cater-12-6.south);%
    \draw[decorate,decoration={raise=12pt}](cater-1-1.west) node[left=12pt] {$x_1$};%
    \draw[decorate,decoration={raise=12pt}](cater-2-1.west) node[left=12pt] {$x_2$};%
    \draw[decorate,decoration={raise=12pt}](cater-3-1.west) node[left=12pt] {$x_3$};%
    \draw[decorate,decoration={raise=12pt}](cater-4-1.west) node[left=12pt] {$x_4$};%
    \draw[decorate,decoration={raise=12pt}](cater-5-1.west) node[left=12pt] {$x_5$};%
    \draw[decorate,decoration={raise=12pt}](cater-7-1.west) node[left=12pt] {$y_1$};%
    \draw[decorate,decoration={raise=12pt}](cater-8-1.west) node[left=12pt] {$y_2$};%
    \draw[decorate,decoration={raise=12pt}](cater-9-1.west) node[left=12pt] {$y_3$};%
    \draw[decorate,decoration={raise=12pt}](cater-10-1.west) node[left=12pt] {$y_4$};%
    \draw[decorate,decoration={raise=12pt}](cater-11-1.west) node[left=12pt] {$y_5$};%
    \draw[decorate,decoration={raise=12pt}](cater-12-1.west) node[left=12pt] {$y_6$};%
\end{tikzpicture}
\caption{A caterpillar $T$ and a matrix whose $8$-row graph is isomorphic to $T$ where the row labeled with $w$ corresponds to the vertex $w$ in $T$.}
\label{fig:caterpillar}
\end{center}
\end{figure}

Let $G$ be the $(t+k)$-row graph of $M$.
We denote the row of $M$ containing the row of $M^*_{t,n}$ corresponding to $x_i$ by $\mathbf{x}_i$ for each $i$, $1\leq i \leq n$, the row of $M$ containing the $i$th row of the $(2,1)$-block of $M$ by $\mathbf{y}_i$ for each $i$, $1\leq i \leq k$.

By the definition of $M^*_{t,n}$, the row of $M^*_{t,n}$ corresponding to $x_i$ and the row of  $M^*_{t,n}$ corresponding to $x_j$ have at most $t-1$ common $1$s if and only if $|j-i|\ge 2$.
Thus $\mathbf{x}_i$ and $\mathbf{x}_j$ have at most $t+k-1$ common $1$s if and only if $|j-i| \ge 2$ and therefore $P$ is an induced subgraph of $G$.

We note that the $i$th row and the $j$th row of the $(2,2)$-block of $M$  have exactly $k-2$ common $1$s for each $i$ and $j$, $1 \le i < j \le k$.
Thus, by (M-1), ${\bf y}_i$ and ${\bf y}_j$ have at most $t+k-1$ common $1$s for each $i$ and $j$, $1 \le i < j \le k$. Thus $y_i$ and $y_j$ are not adjacent in $G$ for each $i$ and $j$, $1 \le i < j \le k$.

We note that the $l$th row of the $(1,2)$-block of $M$ and the $i$th row of the $(2,2)$-block of $M$ have exactly $k-1$ common $1$s for each $i$, $1\leq i \leq k$ and each $l$, $1\leq l \leq n$. Thus, by (M-2), $y_i$ and $x_l$ are adjacent in $T$ if and only if ${\bf y}_i$ and ${\bf x}_l$ have at least $t+k$ common $1$s. Thus we have shown that $T$ is isomorphic to $G$.
Hence $T$ is a $(t+k)$-competition graph by Theorem~\ref{thm:rowgraph}.
\end{proof}

\begin{Lem}\label{lem:star}
Given an integer $n \ge 2$ and the star graph $K_{1,n}$, $\Upsilon(K_{1,n})=[n]$.
\end{Lem}
\begin{proof}
By Corollary~\ref{cor:K_n}, $n+1\not\in \Upsilon(K_{1,n})$, so $\Upsilon(K_{1,n}) \subset [n]$.

Now we show the converse containment.
By Corollary~\ref{cor:forest}, $\{1,2\}\subset \Upsilon(K_{1,n})$.
By Lemma~\ref{lem:induced star}, $n \in \Upsilon(K_{1,n})$.
Therefore $[n]\subset \Upsilon(K_{1,n})$ for $n=2,3$.
Now suppose $n \geq 4$.
Then $\{1,2,n\}\subset \Upsilon(K_{1,n})$ by the above argument.
Thus it is sufficient to show that $p\in \Upsilon(K_{1,n})$ for $3\leq p \leq n-1$.
Now we consider the following matrix $M$ of order $n+1$:
\begin{center}
\begin{tikzpicture}[
    style1/.style={
        matrix of math nodes,
        every node/.append style = {
            text width=#1,
            align=center,
            minimum height=3ex
            },
        nodes in empty cells,
        left delimiter = [,
        right delimiter = ],
    }
]
    \matrix[style1 = 3ex] (add1) { &&&& \\ &&&& \\ &&&& \\&&&& \\&&&& \\};
    \draw[dashed] (add1-1-4.north) -- (add1-5-4.south);
    \draw[dashed] (add1-4-1.west) -- (add1-4-5.east);

    \node at (add1-2-2) {$M_{p-2,n}$};
    \node at (add1-1-5) {$1$};
    \node at (add1-2-5) {$\vdots$};
    \node at (add1-3-5) {$1$};
    \node at (add1-5-1) {$1$};
    \node at (add1-5-2) {$\cdots$};
    \node at (add1-5-3) {$1$};
    \node at (add1-5-5) {$1$};

    \draw[decorate, decoration={raise=12pt}] (add1-3-1.west) node[left=12pt] {$M=$};
\end{tikzpicture}
\end{center}
where $M_{p-2,n}$ is the matrix defined in \eqref{eqn:cyclematrix}.
Let $G$ be the $p$-row graph of $M$. In addition, let $\mathbf{r}_i$ denote the $i$th row of $M$ and $v_i$ be the vertex of $G$ corresponding to $\mathbf{r}_i$ for each $i$, $1 \leq i \leq n+1$.

For each $i=1$, $\ldots$, $n$, the $i$th row of $M_{p-2,n}$ %
contains exactly $p-1$ $1$s by Lemma~\ref{prop:sets}(1).
Thus $\mathbf{r}_i$ contains exactly $p$ $1$s in $M$ and so $v_i$ and $v_{n+1}$ are adjacent in $G$ for each $i= 1,\ldots, n$.

By (2) and (3) of Lemma~\ref{prop:sets},
the $i$th row and the $j$th row of $M_{p-2,n}$ have at most $p-2$ common $1$s for distinct $i$ and $j$ in {$[n]$}.
Therefore $\mathbf{r}_i$ and $\mathbf{r}_j$ have at most $p-1$ common $1$s and so $v_i$ and $v_j$ are not adjacent in $G$ for distinct $i$ and $j$ in {$[n]$}.
Thus $G$ is isomorphic to $K_{1,n}$.
Hence $p \in \Upsilon(K_{1,n})$ and so $[n] \subset \Upsilon(K_{1,n})$.
\end{proof}

\begin{Thm}\label{thm:cater_threshhold}
For a caterpillar $T$ with $n$ vertices,
\[
\Upsilon(T)=
\begin{cases}
  [n-1] & \mbox{if } d(T)=2; \\
  [n-2] & \mbox{if } d(T)=3; \\
  [n-3] & \mbox{if } d(T)\geq 4
\end{cases}
\]
where $d(T)$ denotes the diameter of $T$.
\end{Thm}

\begin{proof}
If $d(T)=2$, then {$T \cong K_{1,n-1}$ and, by Lemma~\ref{lem:star},} $\Upsilon(T)=[n-1]$.

{Suppose} $d(T)\geq 3$.
If $d(T)=3$, $\Upsilon(T) \subset [n-2]$ by Corollary~\ref{cor:K_n} and Proposition~\ref{prop:n-1}.
If $d(T)\geq 4$, $\Upsilon(T) \subset [n-3]$ by Corollary~\ref{cor:n-3}.

To show the converse containment, let $k(T)$ denote the number of vertices which are attached to the spine of $T$.
Now take a positive integer $p \in [n-t]$ where $t=2$ if $d(T)=3$ and $t=3$ if $d(T) \ge 4$.

Since $d(T)$ is the length of the spine of $T$, $n=d(T)+1+k(T)$. Thus $p \le (d(T)+1+k(T))-t$ or
\begin{equation}\label{eqn:p}
p-d(T)+t-1 \le k(T).
\end{equation}
If either $d(T)=3$ and $p\le 2$ or $d(T)\ge 4$ and $p \le d(T)-2$, then the spine of $T$ is a $p$-competition graph by Proposition~\ref{lem:path} and so $T$ is a $p$-competition graph by Theorem~\ref{thm:addVrtx}.

Now assume that one of the following: $d(T)=3$ and $p > 2$; $d(T) \ge 4$ and $p > d(T)-2$. Let
\[\alpha = \begin{cases} p-2 & \mbox{if $d(T)=3$ and $p>2$,}  \\  p-d(T)+2 & \mbox{if $d(T) \ge 4$ and $p > d(T)-2$.} \end{cases}\]
Then $\alpha \geq 1$ and, by \eqref{eqn:p}, we have $\alpha \le k(T)$.
Let $T'$ be a caterpillar obtained from $T$ by deleting some pendent vertices of $T$ so that $d(T')=d(T)$ and $k(T')=\alpha$.
Then, by Lemma~\ref{lem:caterpillar}, $T'$ is a $p$-competition graph and so, by Theorem~\ref{thm:addVrtx}, $T$ is a $p$-competition graph.
\end{proof}

\begin{Cor}
Let $G$ be a graph with $n$ vertices such that
$G/\!\!\sim$ is a caterpillar.
Then
\[
\Upsilon(G)=
\begin{cases}
  [n-1] & \mbox{if } d(G)=2; \\
  [n-2] & \mbox{if } d(G)=3; \\
  [n-3] & \mbox{if } d(G)\geq 4
\end{cases}
\]
where $d(G)$ denotes the diameter of $G$.
\end{Cor}
\begin{proof}
By Theorem~\ref{thm:cater_threshhold},
\[
\Upsilon(G/\!\!\sim)=
\begin{cases}
  [m-1] & \mbox{if } d(G/\!\!\sim)=2; \\
  [m-2] & \mbox{if } d(G/\!\!\sim)=3; \\
  [m-3] & \mbox{if } d(G/\!\!\sim)\geq 4,
\end{cases}
\]
where $m=|V(G/\!\!\sim)|$ and $d(G/\!\!\sim)$ denotes the diameter of $G/\!\!\sim$.
Since $G/\!\!\sim$ has no isolated vertices,
\[
\Upsilon(G) \supset
\begin{cases}
  [n-1] & \mbox{if } d(G)=2; \\
  [n-2] & \mbox{if } d(G)=3; \\
  [n-3] & \mbox{if } d(G)\geq 4,
\end{cases}
\]
by Proposition~\ref{prop:diam} and Corollary~\ref{cor:contraction_pComp}.
By Corollary~\ref{cor:K_n}, $n\not\in\Upsilon(G)$.
Therefore $\Upsilon(G)=[n-1]$ if $d(G)=2$.
By Proposition~\ref{prop:n-1}, $n-1\not\in\Upsilon(G)$ if $d(G)\geq 3$.
Thus $\Upsilon(G)=[n-2]$ if $d(G)=3$.
By Corollary~\ref{cor:n-3}, $\Upsilon(G) \subset [n-3]$ if $d(G) \geq 4$.
Hence $\Upsilon(G)=[n-3]$ if $d(G)\geq 4$.
\end{proof}

\begin{Lem}\label{lem:increasing}
Given a $p$-competition graph $G$ with $n$ vertices,
suppose that {$2^r+1$ neighbors of a vertex $v$ of $G$ form} an independent set for some positive integer $r$.
Then $p \ge n-r$ implies that there are two nonadjacent neighbors $x$ and $y$ of $v$ with $|\Lambda_M(x)| < |\Lambda_M(v)|$ and $|\Lambda_M(y)| < |\Lambda_M(v)|$  for any square matrix $M$ of order $n$ whose $p$-row graph is isomorphic to $G$.
\end{Lem}
\begin{proof}
Let $M$ be a square matrix $M$ of order $n$ whose $p$-row graph is isomorphic to $G$.
Since $v$ is not isolated, $p \le |\Lambda_M(v)| \le n$.
For notational convenience, we let $\overline{\Lambda}_M(v)=[n]\setminus\Lambda_M(v)$.
Now suppose $p \ge n-r$.
Then $0 \le |\overline{\Lambda}_M(v)| \le n-p \le r$.
Thus the number of subsets of $\overline{\Lambda}_M(v)$ is less than $2^{r}+1$.
For each neighbor $w$ of $v$, $\overline{\Lambda}_M(v) \cap \overline{\Lambda}_M(w)$ is a subset of $\overline{\Lambda}_M(v)$.
Since $v$ has  $2^{r}+1$ neighbors which form an independent set  by the hypothesis, there are two nonadjacent neighbors $x$ and $y$ of $v$ such that $\overline{\Lambda}_M(v)  \cap \overline{\Lambda}_M(x) =\overline{\Lambda}_M(v) \cap  \overline{\Lambda}_M(y)$ by the Pigeonhole principle.
Since $\overline{\Lambda}_M(v)  \cap \overline{\Lambda}_M(x)$ and $\overline{\Lambda}_M(v) \cap  \overline{\Lambda}_M(y)$ are subsets of $\overline{\Lambda}_M(x)$ and $\overline{\Lambda}_M(y)$, respectively, we have
\begin{equation}\label{eqn:zero}
\overline{\Lambda}_M(v) \cap \overline{\Lambda}_M(x)={ \overline{\Lambda}_M(v) \cap \overline{\Lambda}_M(y)} \subset \overline{\Lambda}_M(x) \cap \overline{\Lambda}_M(y).
\end{equation}
Since $v$ is adjacent to $x$ and $y$,  $|\overline{\Lambda}_M(v) \cup \overline{\Lambda}_M(x)| \le n-p$ and $|\overline{\Lambda}_M(v) \cup \overline{\Lambda}_M(y)| \le n-p$.
Since $x$ and $y$ are not adjacent, $|\overline{\Lambda}_M(x) \cup \overline{\Lambda}_M(y)| > n-p$.
Thus
\begin{align*}|\overline{\Lambda}_M(v)|+|\overline{\Lambda}_M(x)|-|\overline{\Lambda}_M(v) \cap \overline{\Lambda}_M(x)|  =|\overline{\Lambda}_M(v) \cup \overline{\Lambda}_M(x)| \leq n-p\\ <|\overline{\Lambda}_M(x) \cup \overline{\Lambda}_M(y)|
=|\overline{\Lambda}_M(x)|+|\overline{\Lambda}_M(y)|-|\overline{\Lambda}_M(x) \cap \overline{\Lambda}_M(y)|
\end{align*}
Then, by \eqref{eqn:zero}, $|\overline{\Lambda}_M(y)| > |\overline{\Lambda}_M(v)|$. By the same argument, one can show that  $|\overline{\Lambda}_M(x)| > |\overline{\Lambda}_M(v)|$ and we complete the proof.
\end{proof}

By Proposition~\ref{prop:tree_diam}, $\Upsilon(T)\neq \emptyset$ for a tree $T$ with the diameter at least $3$.
It is easy to see that $\Upsilon(T)\neq \emptyset$ for a tree $T$ with the diameter at most two.
Thus $\max\Upsilon(T)$ exists for any tree $T$.
We have shown that $|V(T)| - \max\Upsilon(T) \leq 3$ for a caterpillar $T$.
One might think by this result that there exists a positive integer $t$ such that $|V(T)| - \max\Upsilon(T) \leq t$ for any tree $T$, yet it is not true by the following theorem.

Let $k$ be a positive integer.
A \emph{$k$-ary tree} is a rooted tree in which each vertex has no more than $k$ children.
A \emph{full $k$-ary tree} is a rooted tree exactly $k$ children or no children.
A \emph{perfect $k$-ary tree} is a full $k$-ary tree in which all pendant vertices are at the same depth.
The \emph{depth} of a vertex in a rooted tree is the distance between the vertex and the root.
The \emph{height} of a rooted tree is the number of edges on the longest path between its root and a pendant vertex.

\begin{Thm}\label{thm:n-r}
For any positive integer $r$, there is a tree $T$ with $|V(T)| - \max\Upsilon(T) > r$.
\end{Thm}
\begin{proof}
Let $T$ be a perfect $(2^r+1)$-ary tree with height $r+1$ and a root $x_0$.
Then by the definition of $\Upsilon(G)$ for a graph $G$, $T$ is a $(\max\Upsilon(T))$-competition graph. 
Then $T$ is a $(\max\Upsilon(T))$-row graph of a matrix $M$.
Suppose that $|V(T)|-\max\Upsilon(T) \leq r$.
Then, by Lemma~\ref{lem:increasing},
$|\Lambda_M(x_1)| < |\Lambda_M(x_0)| \leq |V(T)|$ for some children $x_1$ of $x_0$.
Then, by the same lemma again,
$|\Lambda_M(x_2)| < |\Lambda_M(x_1)| \leq |V(T)|-1$ for some children $x_2$ of $x_1$.
We apply the lemma repeatedly to have $|\Lambda_M(x_{r+1})| < |V(T)|-r$.
Since $x_{r+1}$ is non-isolated, $|\Lambda_M(x_{r+1})| \geq \max\Upsilon(T)$.
Thus $|V(T)| - \max\Upsilon(T) > r$ and we reach a contradiction.
Hence $|V(T)| - \max\Upsilon(T) > r$.
\end{proof}


\section{Closing Remarks}
We have shown that $\Upsilon(K_{3,3})=\emptyset$. We would like to know whether or not $\Upsilon(K_{n,n})=\emptyset$ for any $n \ge 4$.
We have characterized the graphs with $n$ vertices and the competition-realizer $[n]$ and $[n-1]$, respectively.
It would be interesting to characterize the graphs with $n$ vertices and competition-realizer $[n-2]$.
Finally we suggest to find the competition-realizer for a Lobster to extend our result which gives every element in the competition-realizer for a caterpillar.

\section{Acknowledgement}
This research was supported by
the National Research Foundation of Korea(NRF) funded by the Korea government(MSIT) (No.\ NRF-2017R1E1A1A03070489) and by the Korea government(MSIP) (No.\ 2016R1A5A1008055).
\bibliographystyle{plain}

\end{document}